\documentclass[12pt]{amsart}
\usepackage[colorlinks]{hyperref}
\usepackage{amssymb}
\usepackage{amscd}
\usepackage{fullpage}
\usepackage{amsmath}

\numberwithin{equation}{section}
\theoremstyle{plain}
\begin{document}
\newtheorem*{ThA}{Theorem A}
\newtheorem*{ThB}{Theorem B}
\newtheorem*{ThC}{Theorem C}
\newtheorem*{ThD}{Theorem D}
\newtheorem*{definition}{Definition}
\newtheorem{theorem}{Theorem}[section]
\newtheorem{conjecture}[theorem]{Conjecture}
\newtheorem{corollary}[theorem]{Corollary}
\newtheorem{lemma}[theorem]{Lemma}
\newtheorem{proposition}[theorem]{Proposition}
\newtheorem{remark}[theorem]{Remark}
\newtheorem{example}[theorem]{Example}
\newtheorem{open}[theorem]{Open Question}


   \newtheorem{thm}{Theorem}[subsection]
   \newtheorem{cor}[thm]{Corollary}
   \newtheorem{lem}[thm]{Lemma}
   \newtheorem{prop}[thm]{Proposition}
   \newtheorem{defi}[thm]{Definition}
   \newtheorem{rem}[thm]{Remark}
   \newtheorem{exa}[thm]{Example}

\title[Hardy-Hilbert type inequalities]{Hardy-Hilbert type inequalities on homogeneous groups-An introduction and generalization to the kernel case}
\author{M. F. Yimer, L.-E. Persson$^\ast$, M. Ruzhansky, N. Samko and T. G. Ayele}

\address{Department of Mathematics\\ Addis Ababa University, 1176 Addis Ababa, Ethiopia {\em E-mail: markos.fisseha@aau.edu.et} }
\address{Department of Mathematics and Computer Science\\Karlstad University\\ SE-651 88 Karlstad, Sweden\\ and\\ Department of Computer Science and Computational Engineering\\ UiT The Arctic University of Norway, N-8505, Narvik, Norway\\ {\em Email: larserik6pers@gmail.com}}
\address{Department of Mathematics: Analysis, Logic and Discrete Mathematics\\
 Ghent University, Belgium\\ and\\ School of Mathematical Sciences\\Queen Mary University of London\\ United Kingdom\\ {\em E-mail: michael.ruzhansky@ugent.be} }
\address{Department of Computer Science and Computational Engineering\\UiT The Arctic University of Norway, N-8505, Narvik, Norway\\
  {\em Email: nsamko@gmail.com}}
\address{Department of Mathematics\\ Addis Ababa University, 1176 Addis Ababa, Ethiopia {\em E-mail: tsegaye.ayele@aau.edu.et}}


\begin{abstract} 
There is a lot of information available concerning Hardy-Hilbert type inequalities in one or more dimensions. In this paper we introduce the development of such inequalities on homogeneous groups. Moreover, we point out a unification of several of the Hardy-Hilbert type inequalities in the classical case to a general kernel case. Finally, we generalize these results to the homogeneous group case.
\end{abstract}
\maketitle
\allowdisplaybreaks
\section{Introduction}\label{Sec.1}
First, let us recall the well-known classical Hilbert's integral inequality.\\

If $f$ and $g$ are measurable real-valued functions such that $f,\, g\in L^2(0,\infty)$, then the following inequality
\begin{align}\label{F.V.H.H.Eq.1}
\int_0^\infty\int_0^\infty \dfrac{f(x)g(y)}{x+y}\, \mathrm{d}x\mathrm{d}y\leq \pi\left(\int_0^\infty f^2(x) \mathrm{d}x\right)^{1/2}\left(\int_0^\infty g^{2}(y) \mathrm{d}y\right)^{1/2},
\end{align} 
holds, where the constant factor $\pi$ in \eqref{F.V.H.H.Eq.1} is the best possible (see \cite{S}). 

Hilbert worked mostly with the discrete version of \eqref{F.V.H.H.Eq.1}, see his 1906 paper \cite{HI}. The proof of the sharp constant $\pi$ as well as the integral version \eqref{F.V.H.H.Eq.1} was derived by Schur in 1911 (see \cite{WM}). Later, in 1925, Hardy proved an extension of inequality \eqref{F.V.H.H.Eq.1} as follows (see \cite{H1}): 
\begin{ThA}\label{ThmA}
If $p>1$, $q=p/(p-1)$, $f(x)\geq 0$ and $g(x)\geq 0$ are such that $f\in L^p(0,\infty)$ and $g\in L^q(0,\infty)$, then the inequality
\begin{align}\label{H.H.Eq.2}
\int_0^\infty\int_0^\infty \dfrac{f(x)g(y)}{x+y}\, \mathrm{d}x\mathrm{d}y\leq \dfrac{\pi}{\sin \left(\frac{\pi}{p}\right)}\left(\int_0^\infty f^p(x) \mathrm{d}x\right)^{1/p}\left(\int_0^\infty g^{q}(y) \mathrm{d}y\right)^{1/q},
\end{align} 
holds. Moreover, the constant factor $\dfrac{\pi}{\sin\left(\pi/p\right)}$ in \eqref{H.H.Eq.2} is the best possible (see also \cite{HLP}, \cite{KPS}, \cite{P1} and the references therein).
\end{ThA} 

From the history it is known that the main motivation for Hardy to begin the study, which was finalized in 1925 his famous sharp inequality (see \cite{H})
\begin{align*}
\int_0^\infty \left(\dfrac{1}{x}\int_0^x f(t)\mathrm{d}t\right)^p\mathrm{d}x\leq \left(\dfrac{p}{p-1}\right)^p\int_0^\infty f^p(x)\mathrm{d}x,\quad p>1,
\end{align*}
was to find an easy proof of Hilbert's inequality \eqref{F.V.H.H.Eq.1}. For these reasons the inequality \eqref{H.H.Eq.2} is called the Hardy-Hilbert inequality.

It has after that been an almost unbelievable development of what is called Hardy-type inequalities. For this more than 100 years of development we refer to the books \cite{KPS}, \cite{MRDS} and the references therein. We also refer to the recent review paper \cite{PS}.

It has not been the same intensive development of the Hardy-Hilbert type inequalities. However, it has especially during recent years been an increasing interest also in this case and several results dealing with numerous variants, generalizations and extensions have appeared in the literature, see \cite{HLP}, \cite{K}, \cite{LPS2}, \cite{P}, \cite{WM}, \cite{Y1}, \cite{Y} and the references therein. There is now a large number of results generalizing \eqref{H.H.Eq.2} by replacing the kernel $k(x,y)=1/(x+y)$ with some other kernels with similar homogeneity properties. In several cases such results are referred to as Hardy-Hilbert's type inequalities.\\

The main aim of this paper is to initiate the development of Hardy-Hilbert type inequalities on a homogeneous group $\mathbb{G}$ equipped with a quasi-norm $|\cdot|$. For this introduction we have chosen the following two results from the literature mentioned above:

\begin{ThB}\label{H.I.Th.2}
Let $p>1$, $q=p/(p-1)$, and let $f$ and $g$ be nonnegative measurable functions such that $f\in L^p(0,\infty)$ and $g\in L^q(0,\infty)$. Then, the Hilbert integral inequality \eqref{H.H.Eq.2} holds if and only if the following integral inequality holds:
\begin{align}\label{H.H.Eq.6}
\left(\int_0^\infty\left(\int_0^\infty \dfrac{f(x)}{x+y}\, \mathrm{d}x\right)^p\mathrm{d}y\right)^{1/p}\leq \dfrac{\pi}{\sin \left(\frac{\pi}{p}\right)}\left(\int_0^\infty f^p(x) \mathrm{d}x\right)^{1/p}.
\end{align} 
\end{ThB}
For a much more general such equivalence result we refer to \cite[Lemma 7.69]{KPS}.

\begin{ThC}\label{ThemC}
If $p>1, k>1, \frac{1}{p}+\frac{1}{q}=\frac{1}{k}+\frac{1}{m}=1$, and $\lambda>0$, $f$ and $g$ are nonnegative functions such that $0< \int_0^\infty x^{p(1-\frac{\lambda}{k})-1}f^p(x)\mathrm{d}x<\infty$ and $0< \int_0^\infty y^{q(1-\frac{\lambda}{m})-1}g^q(y)\mathrm{d}y<\infty$, then we have
\begin{align}\label{F.H.Th.1.Eq.4}
\begin{split}
&\int_0^\infty\int_0^\infty \dfrac{f(x)g(y)}{x^\lambda+y^\lambda}\, \mathrm{d}x\mathrm{d}y\\
&\leq \dfrac{\pi}{\lambda\sin \left(\frac{\pi}{m}\right)}\left(\int_0^\infty x^{p(1-\frac{\lambda}{k})-1}f^p(x)\, \mathrm{d}x\right)^{1/p}\left(\int_0^\infty y^{q(1-\frac{\lambda}{m})-1}g^q(y) \mathrm{d}y\right)^{1/q},
\end{split}
\end{align}
where the constant factor $\dfrac{\pi}{\lambda\sin\left(\pi/m\right)}$ is the best possible. 
\end{ThC}
In Section \ref{Sec.4} of this paper we rewrite \eqref{F.H.Th.1.Eq.4} in an equivalent form, so that it is known from the general theory (see \cite{LPS2} and the books \cite{KS}, \cite{KPS} and c.f. also \cite{HLP} which shows that already Hardy was aware of this more general idea). We claim that several results in the literature can be put to this general frame, see our Section \ref{Sec.4} especially Theorem \ref{H.K.Thm.4.6}.

The organization and main results of this paper can be described as follows: In order not to disturb our discussions later on we present some preliminaries on homogeneous groups in Section \ref{Sec.2}. In Section \ref{Sec.3} we generalize and unify Theorems A -- C in this homogeneous group frame (see Theorems \ref{H.I.Th.4} and Proposition \ref{H.H.Th.3.5}). We also point out how our Theorem \ref{H.I.Th.4} reads in $\mathbb{R}^n$ on a form we believe is also new (see Corollary \ref{F.H.H.Cor3.3}). To be able to put Proposition \ref{H.H.Th.3.5} (and similar generalizations) to a more general kernel frame we use Section \ref{Sec.4} to present, unify and slightly generalize the classical theory to this general kernel case. Our main Theorem is summarized in Theorem \ref{H.K.Thm.4.6}. Inspired by this theory in Section \ref{Sec.5} we generalize also our Hardy-Hilbert theory on homogeneous groups to a general kernel case. The main results are given in Theorems \ref{G.H.K.Thm.4.1}, \ref{G.H.K.Thm.4.6} and Corollary \ref{G.H.H.Corr.5.6}.  
\section{Preliminaries}\label{Sec.2}
In this section, we recall the basics of homogeneous groups. For further reading on homogeneous groups and other inequalities on homogeneous groups, we refer to the monographs \cite{FR}, \cite{FH}, \cite{MRDS} and references therein. \\

A Lie group $\mathbb{G}$ (identified with $(\mathbb{R}^N, \circ)$) is called a homogeneous group if it is equipped with a dilation mapping
\begin{align*}
D_\lambda: \mathbb{R}^N\to \mathbb{R}^N, \, \lambda>0,
\end{align*}
defined as
\begin{align*}
D_\lambda(x)=(\lambda^{v_1}x_1, \lambda^{v_2}x_2,...,\lambda^{v_N}x_N), v_1,v_2,...,v_N>0,
\end{align*}
which is an automorphism of the group $\mathbb{G}$ for each $\lambda>0$. Here and in the sequel, we will denote the image of $x\in \mathbb{G}$ under $D_\lambda$ by $\lambda(x)$ or, simply $\lambda x$. The homogeneous dimension $Q$ of a homogeneous Lie group $\mathbb{G}$ is defined by 
\begin{align*}
Q=v_1+v_2+\cdots+v_N.
\end{align*}
It is well known that a homogeneous group is necessarily nilpotent and unimodular. There are different particular examples of homogeneous groups such as the Euclidean space $\mathbb{R}^N$ (in which case $Q=N$), the Heisenberg group, as well as general stratified groups (homogeneous Carnot groups) and graded groups.

The Haar measure $\mathrm{d}x$ on $\mathbb{G}$ is nothing but the Lebesgue measure on $\mathbb{R}^N$. \\
Let us denote the volume of a measurable set $\omega\subseteq \mathbb{G}$ by $|\omega|$. Then we have the following consequences: for $\lambda>0$
\begin{align}\label{H.I.Eq.2}
|D_\lambda(\omega)|=\lambda^Q|\omega| \text{ and } \int_{\mathbb{G}} f(\lambda x)\mathrm{d}x=\lambda^{-Q}\int_{\mathbb{G}} f(x)\mathrm{d}x.
\end{align}
\begin{definition}
A quasi-norm on a homogeneous group $\mathbb{G}$ is any continuous function $|\cdot|: \mathbb{G}\to [0, \infty)$ satisfying the following conditions:
\begin{enumerate}
    \item[(i)] $|x|=|x^{-1}|$ for all $x\in \mathbb{G}$
     \item[(ii)] $|\lambda x|=\lambda|x|$ for all $x\in \mathbb{G}$ and $\lambda>0$
      \item[(iii)] $|x|=0 \Longleftrightarrow x=0$.
\end{enumerate}
\end{definition}

Before we finish this section, we need to mention the following {\it polar decomposition} on a homogeneous group $\mathbb{G}$ since it plays an important role in the proofs of our main results, see e.g. \cite[Proposition 1.2.10]{MRDS}
\begin{lemma}
Let 
\begin{align*}
\mathfrak{S}=\{x\in \mathbb{G}: |x|=1\}\subset \mathbb{G}
\end{align*}
be the unit sphere with respect to the quasi-norm $|\cdot|$. Then there is a unique Radon measure $\sigma$ on $\mathfrak{S}$ such that for all $f\in L^1(\mathbb{G})$, 
\begin{align}\label{H.I.Eq.3}
    \int_{\mathbb{G}} f(x)\mathrm{d}x= \int_0^\infty \int_{\mathfrak{S}} f(ry)r^{Q-1} \mathrm{d}\sigma(y)\mathrm{d}r.
\end{align}
\end{lemma}

Here and in the sequel, we use the following notations. A quasi-ball in the homogeneous group $\mathbb{G}$ with radius $|x|$, $x\in \mathbb{G}$, and centered at the origin will be denoted by $B(0, |x|)$. We denote the surface measure of the unit sphere $\mathfrak{S}$ in $\mathbb{G}$ by $|\mathfrak{S}|$.

 The Haar measure of the quasi-ball $B(0, |x|)$ denoted by $|B(0, |x|)|$, can be calculated by using polar decomposition \eqref{H.I.Eq.3} as
\begin{align}\label{H.I.Eq.4}
    |B(0, |x|)|&=\int_{B(0, |x|)} \mathrm{d}y=\int_{0}^{|x|} r^{Q-1}\left(\int_{\mathfrak{S}} \mathrm{d}\sigma(t)\right)\mathrm{d}r=\dfrac{|\mathfrak{S}|}{Q}|x|^Q.
\end{align}

\section{Some introductory Hardy-Hilbert type inequalities on homogeneous groups}\label{Sec.3}
Our first main result is the following unification and generalization of Theorems A and B:
\begin{theorem}\label{H.I.Th.4}
Let $\mathbb{G}$ be a homogeneous group with the homogeneous dimension $Q$ equipped with a quasi-norm $|\cdot|$. Let $p>1$, $q=p/(p-1)$, and let $f$ and $g$ be nonnegative and measurable functions on $\mathbb{G}$. Moreover, assume that $F(r)\in L^p(0,\infty)$ and $G(s)\in L^q(0,\infty)$, where
\begin{align}\label{H.H.Eq.7.1}
F(r)=r^\frac{Q-1}{p}\int_\mathfrak{S} f(rx)\mathrm{d}\sigma(x) \text{ and } G(s)=s^\frac{Q-1}{q}\int_\mathfrak{S} g(sy)\mathrm{d}\sigma(y).
\end{align}
Then, the following Hardy-Hilbert type inequality holds:
\begin{align}\label{H.H.Eq.7}
\begin{split}
\int_\mathbb{G}\int_\mathbb{G} \dfrac{f(x)g(y)}{|x|+|y|} \left(\dfrac{|x|}{|B(0,|x|)|}\right)^{1/q}\left(\dfrac{|y|}{|B(0,|y|)|}\right)^{1/p} \mathrm{d}x\mathrm{d}y\\
\leq \dfrac{Q \pi}{\sin \left(\frac{\pi}{p}\right)} \left(\int_\mathbb{G} f^p(x) \mathrm{d}x\right)^{1/p}\left(\int_\mathbb{G} g^{q}(y) \mathrm{d}y\right)^{1/q}.
\end{split}
\end{align}

Moreover, the following integral inequalities hold and they are both equivalent to \eqref{H.H.Eq.7}:
\begin{align}\label{H.H.Eq.11}
\begin{split}
\left(\int_\mathbb{G}\left(\int_\mathbb{G} \dfrac{f(x)}{|x|+|y|} \left(\dfrac{|x|}{|B(0,|x|)|}\right)^{1/q} \mathrm{d}x\right)^p\dfrac{|y|}{|B(0,|y|)|} \mathrm{d}y\right)^\frac{1}{p}\\
\leq \dfrac{Q \pi}{\sin \left(\frac{\pi}{p}\right)} \left(\int_\mathbb{G} f^p(x) \mathrm{d}x\right)^{1/p},
\end{split}
\end{align}
and 
\begin{align}\label{H.H.Eq.112}
\begin{split}
\left(\int_\mathbb{G}\left(\int_\mathbb{G} \dfrac{g(y)}{|x|+|y|} \left(\dfrac{|y|}{|B(0,|y|)|}\right)^{1/p} \mathrm{d}y\right)^q\dfrac{|x|}{|B(0,|x|)|} \mathrm{d}x\right)^\frac{1}{q}\\
\leq \dfrac{Q \pi}{\sin \left(\frac{\pi}{p}\right)} \left(\int_\mathbb{G} g^q(y) \mathrm{d}y\right)^{1/q}.
\end{split}
\end{align}
The constant $\dfrac{Q \pi}{\sin\left(\pi/p\right)}$ is sharp in all inequalities \eqref{H.H.Eq.7} -- \eqref{H.H.Eq.112}.
\end{theorem}
\begin{proof}
By using polar decomposition \eqref{H.I.Eq.3} and \eqref{H.I.Eq.4} on a homogeneous group $\mathbb{G}$, we have
\begin{align*}
I:&=\int_\mathbb{G}\int_\mathbb{G} \dfrac{f(x)g(y)}{|x|+|y|} \left(\dfrac{|x|}{|B(0,|x|)|}\right)^{1/q}\left(\dfrac{|y|}{|B(0,|y|)|}\right)^{1/p} \mathrm{d}x\mathrm{d}y\\
&=\dfrac{Q}{|\mathfrak{S}|}\int_0^\infty\int_0^\infty \dfrac{F(r)G(s)}{r+s}\mathrm{d}r\mathrm{d}s,
\end{align*}
where $F(r)$ and $G(s)$ are defined by \eqref{H.H.Eq.7.1}. Then, by Hardy-Hilbert's inequality \eqref{H.H.Eq.2}, Jensen's inequality and again \eqref{H.I.Eq.3}, we obtain that
\begin{align*}
I\leq \dfrac{Q}{|\mathfrak{S}|} \dfrac{\pi}{\sin \left(\frac{\pi}{p}\right)} &\left(\int_0^\infty F^p(r)\,\mathrm{d}r\right)^{1/p}\left(\int_0^\infty G^{q}(s)\,\mathrm{d}s\right)^{1/q}\\
=Q|\mathfrak{S}| \dfrac{\pi}{\sin \left(\frac{\pi}{p}\right)}& \left(\int_0^\infty r^{Q-1} \left(\dfrac{1}{|\mathfrak{S}|}\int_\mathfrak{S} f(rx)\mathrm{d}\sigma(x)\right)^p\,\mathrm{d}r\right)^{1/p}\times\\
&\left(\int_0^\infty s^{Q-1}\left(\dfrac{1}{|\mathfrak{S}|}\int_\mathfrak{S} g(sy)\mathrm{d}\sigma(y)\right)^{q}\,\mathrm{d}s\right)^{1/q}\\
\leq Q|\mathfrak{S}| \dfrac{\pi}{\sin \left(\frac{\pi}{p}\right)}& \left(\int_0^\infty r^{Q-1} \left(\dfrac{1}{|\mathfrak{S}|}\int_\mathfrak{S} f^p(rx)\mathrm{d}\sigma(x)\right)\,\mathrm{d}r\right)^{1/p}\times\\
&\left(\int_0^\infty s^{Q-1}\left(\dfrac{1}{|\mathfrak{S}|}\int_\mathfrak{S} g^{q}(sy)\mathrm{d}\sigma(y)\right)\,\mathrm{d}s\right)^{1/q}\\
=\dfrac{Q\pi}{\sin \left(\frac{\pi}{p}\right)}&\left(\int_\mathbb{G} f^p(x) \mathrm{d}x\right)^{1/p}\left(\int_\mathbb{G} g^{q}(y) \mathrm{d}y\right)^{1/q}.
\end{align*}
Therefore, \eqref{H.H.Eq.7} holds and the sharpness of the constant $\dfrac{Q \pi}{\sin\left(\pi/p\right)}$ in \eqref{H.H.Eq.7} is guaranteed from the sharpness of the constant in \eqref{H.H.Eq.2} and Jensen's inequality.

Next we assume that \eqref{H.H.Eq.7} holds and use it with 
\begin{align*}
g(y)=\left(\int_\mathbb{G} \dfrac{f(x)}{|x|+|y|} \left(\dfrac{|x|}{|B(0,|x|)|}\right)^{1/q}\mathrm{d}x\right)^{p-1} \left(\dfrac{|y|}{|B(0,|y|)|}\right)^{1/q}.
\end{align*}
Then we find that
\begin{align*}
&\int_\mathbb{G}\left(\int_\mathbb{G} \dfrac{f(x)}{|x|+|y|} \left(\dfrac{|x|}{|B(0,|x|)|}\right)^{1/q} \mathrm{d}x\right)^p\dfrac{|y|}{|B(0,|y|)|} \mathrm{d}y\\
&=\int_\mathbb{G}\int_\mathbb{G} \dfrac{f(x)g(y)}{|x|+|y|} \left(\dfrac{|x|}{|B(0,|x|)|}\right)^{1/q}\left(\dfrac{|y|}{|B(0,|y|)|}\right)^{1/p} \mathrm{d}x\mathrm{d}y\\
&\leq \dfrac{Q \pi}{\sin \left(\frac{\pi}{p}\right)} \left(\int_\mathbb{G} f^p(x) \mathrm{d}x\right)^{1/p}\left(\int_\mathbb{G} g^{q}(y) \mathrm{d}y\right)^{1/q}\\
&= \dfrac{Q \pi}{\sin \left(\frac{\pi}{p}\right)} \left(\int_\mathbb{G} f^p(x) \mathrm{d}x\right)^{1/p}\left[\int_\mathbb{G}\left(\int_\mathbb{G} \dfrac{f(x)}{|x|+|y|} \left(\dfrac{|x|}{|B(0,|x|)|}\right)^{1/q}\mathrm{d}x\right)^p\dfrac{|y|}{|B(0,|y|)|}\mathrm{d}y\right]^{1/q},
\end{align*}
and since $1-1/q=1/p$ we conclude that \eqref{H.H.Eq.11} holds.

 Conversely, we assume that \eqref{H.H.Eq.11} holds. Then, by H$\ddot{\text{o}}$lder's inequality,
\begin{align*}
&\int_\mathbb{G}\int_\mathbb{G} \dfrac{f(x)g(y)}{|x|+|y|} \left(\dfrac{|x|}{|B(0,|x|)|}\right)^{1/q}\left(\dfrac{|y|}{|B(0,|y|)|}\right)^{1/p}\mathrm{d}x\mathrm{d}y\\
&= \int_\mathbb{G} g(y)\left(\left(\dfrac{|y|}{|B(0,|y|)|}\right)^{1/p}\int_\mathbb{G} \dfrac{f(x)}{|x|+|y|} \left(\dfrac{|x|}{|B(0,|x|)|}\right)^{1/q}\mathrm{d}x\right)\mathrm{d}y\\
&\leq \left(\int_\mathbb{G} g^{q}(y)\,\mathrm{d}y\right)^{1/q}\left[\int_\mathbb{G}\left(\int_\mathbb{G} \dfrac{f(x)}{|x|+|y|} \left(\dfrac{|x|}{|B(0,|x|)|}\right)^{1/q}\mathrm{d}x\right)^p\dfrac{|y|}{|B(0,|y|)|}\,\mathrm{d}y\right]^{1/p}\\
&\leq \dfrac{Q \pi}{\sin \left(\frac{\pi}{p}\right)} \left(\int_\mathbb{G} f^p(x) \mathrm{d}x\right)^\frac{1}{p}\left(\int_\mathbb{G} g^{q}(y)\,\mathrm{d}y\right)^{1/q},
\end{align*}
so \eqref{H.H.Eq.7} holds. The proof of the equivalence between \eqref{H.H.Eq.7} and \eqref{H.H.Eq.112} can be done analogously. But, simply, the equivalence between the inequalities \eqref{H.H.Eq.11} and \eqref{H.H.Eq.112} follows by changing notations and use the symmetry (e.g. that $\frac{1}{p}+\frac{1}{q}=1$ so $\sin(\frac{\pi}{q})=\sin(\frac{\pi}{p})$). The proof is complete. 
\end{proof}
\begin{corollary}\label{F.H.H.Cor3.3}
Let $p>1, q=p/(p-1)$ and let $f$ and $g$ be nonnegative and measurable functions on $\mathbb{R}^n, n\in \mathbb{Z}_+$, such that $f\in L^p(\mathbb{R}^n)$ and $g\in L^q(\mathbb{R}^n)$. Then, the following Hardy-Hilbert type inequality holds:
\begin{align}\label{F.H.H.Cor.3.3.Equ.3.5}
\begin{split}
\int_{\mathbb{R}^n}\int_{\mathbb{R}^n} \dfrac{f(x)g(y)}{|x|+|y|}|x|^{(1-n)\frac{1}{q}}|y|^{(1-n)\frac{1}{p}}\,\mathrm{d}x\mathrm{d}y\\
\leq \dfrac{n\pi}{\sin \frac{\pi}{p}}\dfrac{\pi^{n/2}}{\Gamma(\frac{n}{2}+1)}\left(\int_{\mathbb{R}^n} f^p(x)\,\mathrm{d}x\right)^{1/p}\left(\int_{\mathbb{R}^n} g^q(y)\,\mathrm{d}y\right)^{1/q}.
\end{split}
\end{align}
Moreover, the following integral inequalities hold and are equivalent to \eqref{F.H.H.Cor.3.3.Equ.3.5}:
\begin{align}\label{F.H.H.Cor.3.3.Equ.3.6}
\left(\int_{\mathbb{R}^n}\left(\int_{\mathbb{R}^n} \dfrac{f(x)}{|x|+|y|}|x|^{(1-n)\frac{1}{q}}\,\mathrm{d}x\right)^p |y|^{1-n}\mathrm{d}y\right)^{1/p}\leq \dfrac{n\pi}{\sin \frac{\pi}{p}}\dfrac{\pi^{n/2}}{\Gamma(\frac{n}{2}+1)}\left(\int_{\mathbb{R}^n} f^p(x)\,\mathrm{d}x\right)^{1/p}
\end{align}
and
\begin{align}\label{F.H.H.Cor.3.3.Equ.3.7}
\left(\int_{\mathbb{R}^n}\left(\int_{\mathbb{R}^n} \dfrac{g(y)}{|x|+|y|}|y|^{(1-n)\frac{1}{p}}\,\mathrm{d}y\right)^q |x|^{1-n}\mathrm{d}x\right)^{1/q}\leq \dfrac{n\pi}{\sin \frac{\pi}{p}}\dfrac{\pi^{n/2}}{\Gamma(\frac{n}{2}+1)}\left(\int_{\mathbb{R}^n} g^q(y)\,\mathrm{d}y\right)^{1/q}.
\end{align}
The constant $\dfrac{n\pi}{\sin \frac{\pi}{p}}\dfrac{\pi^{n/2}}{\Gamma(\frac{n}{2}+1)}$ is sharp in all three inequalities \eqref{F.H.H.Cor.3.3.Equ.3.5} -- \eqref{F.H.H.Cor.3.3.Equ.3.7}. 
\end{corollary}
$\Gamma=\Gamma(x)$ denotes the usual Gamma function in Corollary \ref{F.H.H.Cor3.3}.
\begin{proof}
Use Theorem \ref{H.I.Th.4} for the case when $G=\mathbb{R}^n$ so that $Q=n$ and $B(0, |x|)$ stands for the volume of a ball with radius $|x|$. Therefore, it is known that in this case 
\begin{align*}
B(0,|x|)=\dfrac{\pi^{n/2}}{\Gamma(\frac{n}{2}+1)}|x|^n.
\end{align*}
So the calculation of the sharp constant only consist of some standard calculations.
\end{proof}
\begin{remark}
Note that the ``kernel'' 
\begin{align*}
\dfrac{|x|^{(1-n)\frac{1}{q}}|y|^{(1-n)\frac{1}{p}}}{|x|+|y|}
\end{align*}
in \eqref{F.H.H.Cor.3.3.Equ.3.5} is homogeneous of degree $-n$ so it also follows from the general theory that \eqref{F.H.H.Cor.3.3.Equ.3.5} holds with some constant but with our technique we also get the sharp constant $\dfrac{n\pi}{\sin \frac{\pi}{p}}\dfrac{\pi^{n/2}}{\Gamma(\frac{n}{2}+1)}$. The equivalence of \eqref{F.H.H.Cor.3.3.Equ.3.5} -- \eqref{F.H.H.Cor.3.3.Equ.3.7} follows also from standard arguments we explain more in detail in next two sections. See also the books \cite{KS} and \cite{KPS}.
\end{remark}
By using the technique above in the proof of \eqref{H.H.Eq.7} we can make a similar generalization of several Hardy-Hilbert type inequalities in the literature with other kernels $k(x,y)$ than the classical Hilbert kernel $k(x,y)=1/(x+y)$. As one example with the kernel $1/(x^\lambda+y^\lambda)$ in \eqref{F.H.Th.1.Eq.4} we can generalize Theorem C as follows:
\begin{proposition}\label{H.H.Th.3.5}
Let $\mathbb{G}$ be a homogeneous group with the homogeneous dimension $Q$ equipped with a quasi-norm $|\cdot|$. If $p>1, k>1, \frac{1}{p}+\frac{1}{q}=\frac{1}{k}+\frac{1}{m}=1$, and $\lambda>0$, $f$ and $g$ are nonnegative functions such that $0< \int_0^\infty r^{p(1-\frac{\lambda}{k})-1}F^p(r)\mathrm{d}r<\infty$ and $0< \int_0^\infty s^{q(1-\frac{\lambda}{m})-1}G^q(s)\mathrm{d}s<\infty$, then the inequality
\begin{align}\label{F.H.H.P.1.Eq.24}
\begin{split}
\int_\mathbb{G}\int_\mathbb{G} \dfrac{f(x)g(y)}{|x|^\lambda+|y|^\lambda} \left(\dfrac{|x|}{|B(0,|x|)|}\right)^{1/q}\left(\dfrac{|y|}{|B(0,|y|)|}\right)^{1/p} \mathrm{d}x\mathrm{d}y\\
\leq\dfrac{Q \pi}{\lambda\sin \left(\frac{\pi}{m}\right)} \left(\int_\mathbb{G} |x|^{p(1-\frac{\lambda}{k})-1}f^p(x) \mathrm{d}x\right)^{1/p}\left(\int_\mathbb{G} |y|^{q(1-\frac{\lambda}{m})-1}g^{q}(y) \mathrm{d}y\right)^{1/q},
\end{split}
\end{align}
holds, where $F(r)$ and $G(s)$ are defined by \eqref{H.H.Eq.7.1}.
Moreover, the constant factor $\dfrac{Q \pi}{\lambda\sin\left(\pi/k\right)}$ in \eqref{F.H.H.P.1.Eq.24} is the best possible. 
\end{proposition}
The proof of Proposition \ref{H.H.Th.3.5} is analogous to that of Theorem \ref{H.I.Th.4} but guided by the theory in the classical case (see Section \ref{Sec.4}), we choose another strategy namely to directly generalize Hardy-Hilbert theory to the general kernel case so for example Proposition \ref{H.H.Th.3.5} (and other similar generalizations of classical results) are just special cases of our general result. In the next section we present the motivation for this strategy.

\section{On the general kernel case: the classical situation}\label{Sec.4}
By a kernel we mean a positive and measurable function $k(x,y)$ on $\mathbb{R}_+\times\mathbb{R}_+$, we say that the kernel is homogeneous of order $\lambda$ if 
\begin{align*}
k(ax, ay)=a^\lambda k(x,y),\quad x,y\in \mathbb{R}_+, a>0.
\end{align*}
The kernels of order $-1$ are of particular importance for the general theory. For such kernels we define the constant
\begin{align}\label{H.K.C.Equ.4.2}
C^\ast_p=\int_0^\infty k(1,y)y^{-\frac{1}{p}}\,\mathrm{d}y=\int_0^\infty k(x,1)x^{-\frac{1}{q}}\,\mathrm{d}x,
\end{align}
where $p>1, \dfrac{1}{p}+\dfrac{1}{q}=1.$

The following Theorem is known (see \cite[Theorem 7.48]{KPS}):
\begin{theorem}\label{H.K.Thm.4.1}
Let the kernel $k(x,y)$ satisfy \eqref{H.K.C.Equ.4.2} and $p>1$. Then the inequality 
\begin{align}\label{H.K.Equ.4.3}
\int_0^\infty\left(\int_0^\infty k(x,y)f(y)\,\mathrm{d}y\right)^p\mathrm{d}x\leq C\int_0^\infty f^p(x)\,\mathrm{d}x
\end{align}
holds for all positive and measurable functions on $\mathbb{R}_+$ if and only if $C_p^\ast <\infty$. Moreover, the sharp constant in \eqref{H.K.Equ.4.3} is $C=(C_p^\ast)^p$.
\end{theorem}
\begin{remark}
A simple proof of Theorem \ref{H.K.Thm.4.1} can be found in \cite{LPS2}, but the result was known before, see \cite[Theorem 319]{HLP} and even in multidimensional cases in the works of N. Karapetiants and S. Samko (see their book \cite{KS} and the references therein). 
\end{remark}
\begin{example}
In Theorem C we replace $x^{p(1-\frac{\lambda}{k})-1}f^p(x)$ by $f^p(x)$ and $y^{q(1-\frac{\lambda}{m})-1}g^q(y)$ by $g^q(y)$ and we find that \eqref{F.H.Th.1.Eq.4} can equivalently be written as 
\begin{align*}
\begin{split}
&\int_0^\infty\int_0^\infty \dfrac{x^{-1+\frac{\lambda}{k}+\frac{1}{p}} y^{-1+\frac{\lambda}{m}+\frac{1}{q}}}{x^\lambda+y^\lambda}f(x)g(y)\, \mathrm{d}x\mathrm{d}y\\
&\leq \dfrac{\pi}{\lambda\sin \left(\frac{\pi}{m}\right)}\left(\int_0^\infty f^p(x)\, \mathrm{d}x\right)^{1/p}\left(\int_0^\infty g^q(y) \mathrm{d}y\right)^{1/q}.
\end{split}
\end{align*}
\end{example}
Hence, since the kernel $k(x,y)= \dfrac{x^{-1+\frac{\lambda}{k}+\frac{1}{p}} y^{-1+\frac{\lambda}{m}+\frac{1}{q}}}{x^\lambda+y^\lambda}$ is obviously homogeneous of degree $-1$ we can conclude that Theorem C is a simple special case of Theorem \ref{H.K.Thm.4.1}.\\

Another important remark is:
\begin{remark}\label{H.Rem.4.4}
By using a standard dilation argument we see that among homogeneous kernels $k(x,y)$ \eqref{H.K.Equ.4.3} can never hold for any $\lambda\neq -1$ so the condition $\lambda=-1$ in Theorem \ref{H.K.Thm.4.1} is also necessary. See \cite[Remark 7.68]{KPS}.
\end{remark}
Another useful equivalence result holds for all kernels.

\begin{lemma}
Let $p>1, \dfrac{1}{p}+\dfrac{1}{q}=1$, $C>0$ and $k(x,y)$ be a general kernel. Then the following inequalities are equivalent:
\begin{align}
\int_0^\infty\int_0^\infty k(x,y)f(x)g(y)\mathrm{d}x\mathrm{d}y&\leq C\left(\int_0^\infty f^p(x)\mathrm{d}x\right)^\frac{1}{p}\left(\int_0^\infty g^q(y)\mathrm{d}y\right)^\frac{1}{q}\label{H.K.Lem4.5.Equ.4.4}\\
\int_0^\infty\left(\int_0^\infty k(x,y)f(x)\mathrm{d}x\right)^p\mathrm{d}y&\leq C^p\int_0^\infty f^p(x)\mathrm{d}x\label{H.K.Lem4.5.Equ.4.5}\\
\int_0^\infty\left(\int_0^\infty k(x,y)g(y)\mathrm{d}y\right)^q\mathrm{d}x&\leq C^q\int_0^\infty g^q(y)\mathrm{d}y.\label{H.K.Lem4.5.Equ.4.6}
\end{align}
\end{lemma}
\begin{proof}
For the proof of the equivalence between \eqref{H.K.Lem4.5.Equ.4.4} and \eqref{H.K.Lem4.5.Equ.4.5} see Lemma 7.69 in \cite{KPS}. The proof of the equivalence of \eqref{H.K.Lem4.5.Equ.4.4} and \eqref{H.K.Lem4.5.Equ.4.6} is similar but for the reader's convenience we give the details:

Let \eqref{H.K.Lem4.5.Equ.4.6} hold. Then, by H$\ddot{\text{o}}$lder's inequality
\begin{align*}
I_1:&=\int_0^\infty\int_0^\infty k(x,y)f(x)g(y)\mathrm{d}x\mathrm{d}y\\
&\leq\left( \int_0^\infty\left(\int_0^\infty k(x,y)g(y)\mathrm{d}y\right)^q\mathrm{d}x\right)^\frac{1}{q}\left(\int_0^\infty f^p(x)\mathrm{d}x\right)^\frac{1}{p}\\
&\leq C\left(\int_0^\infty g^q(y)\mathrm{d}y\right)^\frac{1}{q} \left(\int_0^\infty f^p(x)\mathrm{d}x\right)^\frac{1}{p}
\end{align*}
so inequality \eqref{H.K.Lem4.5.Equ.4.4} holds. Conversely, assume that \eqref{H.K.Lem4.5.Equ.4.4} holds and apply it with
\begin{align*}
f(x)=\left(\int_0^\infty k(x,y)g(y)\mathrm{d}y\right)^{q-1}\in L^p.
\end{align*}
Then, by using \eqref{H.K.Lem4.5.Equ.4.4}
\begin{align*}
I_1&=\int_0^\infty\left(\int_0^\infty k(x,y)g(y)\mathrm{d}y\right)^q\mathrm{d}x\\
&\leq C\left(\int_0^\infty\left(\int_0^\infty k(x,y)g(y)\mathrm{d}y\right)^{(q-1)p}\mathrm{d}x\right)^\frac{1}{p}\left(\int_0^\infty g^q(y)\mathrm{d}y\right)^\frac{1}{q}\\
&=CI_1^\frac{1}{p}\left(\int_0^\infty g^q(y)\mathrm{d}y\right)^\frac{1}{q},
\end{align*}
so \eqref{H.K.Lem4.5.Equ.4.6} holds because $1-\dfrac{1}{p}=\dfrac{1}{q}$. The proof is complete.
\end{proof}

Summing up our investigations above we have the following equivalence Theorem connected to Hardy-Hilbert type inequalities with kernels.
\begin{theorem}\label{H.K.Thm.4.6}
Let $p>1$, the kernel $k(x,y)$ be homogeneous of order $-1$ and the constant $C_p^\ast$ be defined by \eqref{H.K.C.Equ.4.2}. Then the following four statements are equivalent
\begin{enumerate}
\item[(1)] The constant $C_p^\ast <\infty$.
\item[(2)] The Hardy-Hilbert type inequality \eqref{H.K.Lem4.5.Equ.4.4} holds.
\item[(3)] The inequality \eqref{H.K.Lem4.5.Equ.4.5} holds.
\item[(4)] The inequality \eqref{H.K.Lem4.5.Equ.4.6} holds.
\end{enumerate}
Moreover, the constant $C_p^\ast$ is sharp in all of \eqref{H.K.Lem4.5.Equ.4.4} -- \eqref{H.K.Lem4.5.Equ.4.6}.
\end{theorem}
\begin{open}
Find necessary and sufficient conditions on a general kernel $k(x,y)$ so that the Hardy-Hilbert type inequality \eqref{H.K.Lem4.5.Equ.4.4} (and, thus, the equivalent inequality \eqref{H.K.Lem4.5.Equ.4.5} and \eqref{H.K.Lem4.5.Equ.4.6}) holds.
\end{open}

\begin{remark}
For the case with homogeneous kernels we have a satisfactory answer in our Theorem \ref{H.K.Thm.4.6} combined with Remark \ref{H.Rem.4.4}.
\end{remark}

Guided by the investigations in this section we shall finalize this paper by investigating the corresponding theory also in the homogeneous group situation also in this most general kernel situation.
\section{The general kernel case: the homogeneous group situation}\label{Sec.5}
In this case, by a kernel we mean a non-negative function $k(|x|,|y|)$ on $\mathbb{G}\times\mathbb{G}$. We say that the kernel is homogeneous of order $\lambda$ if 
\begin{align*}
k\big(a|x|, a|y|\big)=a^\lambda k\big(|x|,|y|\big),\quad x,y\in \mathbb{G}, a>0.
\end{align*}
First we state the following Lemma we need later on but which is also of independent interest:
\begin{lemma}\label{G.H.K.Lem.5.1}
Let $k(|x|,|y|)$ be a homogeneous kernel of degree $-Q$ on $\mathbb{G}\times\mathbb{G}$. If $p>1, \dfrac{1}{p}+\dfrac{1}{q}=1$, then
\begin{align*}
\int_\mathbb{G} k(1,|y|) |y|^{-\frac{Q}{p}}\,\mathrm{d}y=\int_\mathbb{G} k(|x|,1) |x|^{-\frac{Q}{q}}\,\mathrm{d}x.
\end{align*}
\end{lemma}
\begin{proof}
Let $k(|x|,|y|)$ be a homogeneous kernel of degree $-Q$ on $\mathbb{G}\times\mathbb{G}$. Then, using polar decomposition \eqref{H.I.Eq.3}, we have that
\begin{align}\label{G.H.K.K.Lem5.1.Equ.5.2}
\int_\mathbb{G} k(1,|y|) |y|^{-\frac{Q}{p}}\,\mathrm{d}y&=\int_0^\infty\int_{\mathfrak{S}} k(1,s) s^{Q-1-Q/p}\mathrm{d}\sigma(\xi)\mathrm{d}s \nonumber\\
&=\int_0^\infty\int_{\mathfrak{S}} k(1/s,1) s^{-1-Q/p}\mathrm{d}\sigma(\xi)\mathrm{d}s.
\end{align}
Let $r=1/s$. Then, from \eqref{G.H.K.K.Lem5.1.Equ.5.2}, we have that
\begin{align*}
\int_\mathbb{G} k(1,|y|) |y|^{-\frac{Q}{p}}\,\mathrm{d}y&=\int_0^\infty\int_{\mathfrak{S}} k(r,1) r^{Q-1-Q/q}\,\mathrm{d}\sigma(\xi)\mathrm{d}r \nonumber\\
&=\int_\mathbb{G} k(|x|,1) |x|^{-Q/q}\,\mathrm{d}x.
\end{align*}
This completes the proof.
\end{proof}
The kernels of order $-Q$ is of particular importance for the general theory. For such kernels we define the constant (c.f. Lemma \ref{G.H.K.Lem.5.1})
\begin{align}\label{G.H.K.C.Equ.4.2}
C^\ast_p=\int_\mathbb{G} k(1,|y|) |y|^{-\frac{Q}{p}}\,\mathrm{d}y=\int_\mathbb{G} k(|x|,1) |x|^{-\frac{Q}{q}}\,\mathrm{d}x,
\end{align}
where $p>1, \dfrac{1}{p}+\dfrac{1}{q}=1$.

Our first main result in this section is the following generalization of Theorem \ref{H.K.Thm.4.1}:
\begin{theorem}\label{G.H.K.Thm.4.1}
Let the kernel $k(|x|,|y|)$ satisfy \eqref{G.H.K.C.Equ.4.2} and $p>1$. Then the inequality 
\begin{align}\label{G.H.K.Equ.4.3}
\left(\int_\mathbb{G}\left(\int_\mathbb{G} k(|x|,|y|)f(y)\,\mathrm{d}y\right)^p\mathrm{d}x\right)^{1/p}\leq C\left(\int_\mathbb{G} f^p(x)\,\mathrm{d}x\right)^{1/p}
\end{align}
holds for all positive and measurable functions $f$ on $\mathbb{G}$ if and only if
\begin{align}\label{G.H.H.K.N.Equ.1.4}
C_p^\ast <\infty.
\end{align}
Moreover, the sharp constant in \eqref{G.H.K.Equ.4.3} is $C=C_p^\ast$.
\end{theorem}
\begin{proof}
{\it Sufficiency.} Suppose that \eqref{G.H.H.K.N.Equ.1.4} holds. By Jensen's inequality and then Fubini's theorem, we have that
\begin{align}\label{G.H.H.K.N.Equ.1.6}
\int_\mathbb{G}\left(\int_\mathbb{G} k(|x|,|y|)f(y)\,\mathrm{d}y\right)^p\mathrm{d}x\nonumber\\=\left(C_p^\ast\right)^p\int_0^\infty\int_{\mathfrak{S}}\left(\dfrac{1}{C_p^\ast}\int_0^\infty\int_{\mathfrak{S}} k(1,t)t^{Q-1}t^{-Q/p}\cdot t^{Q/p}f(tr\xi)\mathrm{d}\sigma(\xi)\mathrm{d}t\right)^p r^{Q-1} \mathrm{d}\sigma(\eta)\mathrm{d}r\nonumber\\
\leq \left(C_p^\ast\right)^{p-1} \int_\mathbb{G} \left(\int_\mathbb{G} k(|x|,|y|) |y|^{Q/q} f^p(y)\,\mathrm{d}y\right) |x|^{-Q/q}\,\mathrm{d}x\nonumber\\
=\left(C_p^\ast\right)^{p-1} \int_\mathbb{G} \left(\int_\mathbb{G} k(|x|,|y|) |x|^{-Q/q}\,\mathrm{d}x\right) |y|^{Q/q} f^p(y)\,\mathrm{d}y\nonumber\\
=\left(C_p^\ast\right)^{p-1}\int_\mathbb{G} \left(\int_\mathbb{G} k(|x|,1) |x|^{-Q/q} \,\mathrm{d}x\right) f^p(y)\,\mathrm{d}y\nonumber\\
=\left(C_p^\ast\right)^{p}\int_\mathbb{G} f^p(y)\,\mathrm{d}y,
\end{align}
so that \eqref{G.H.K.Equ.4.3} follows from \eqref{G.H.H.K.N.Equ.1.4} and \eqref{G.H.H.K.N.Equ.1.6}. Moreover, the best constant $C$ in \eqref{G.H.K.Equ.4.3} satisfies
\begin{align*}
C\leq C_p^\ast.
\end{align*}

{\it Necessity.} To prove that \eqref{G.H.K.Equ.4.3} implies \eqref{G.H.H.K.N.Equ.1.4}, we define the test function $f$ by
\begin{align*}
f(x)=|x|^{-\frac{Q}{p}-\beta}\chi_{(1,\infty)}(|x|),
\end{align*}
for $\beta>0$. Then, the right hand side of \eqref{G.H.K.Equ.4.3} becomes
\begin{align}\label{G.H.K.Equ.4.8}
\left(\int_\mathbb{G} f^p(x)\,\mathrm{d}x\right)^{1/p}&=\left(\int_0^\infty \int_{\mathfrak{S}} r^{Q-1} f^p(r\xi)\,\mathrm{d}\sigma(\xi)\mathrm{d}r\right)^{1/p}\nonumber\\
&= \left(\int_1^\infty \int_{\mathfrak{S}} r^{-1-\beta p} \,\mathrm{d}\sigma(\xi)\mathrm{d}r\right)^{1/p}\nonumber\\
&=\left(\dfrac{|\mathfrak{S}|}{\beta p}\right)^{1/p}.
\end{align}
On the other hand, the left hand side of \eqref{G.H.K.Equ.4.3} becomes
\begin{align*}
\left(\int_\mathbb{G}\left(\int_\mathbb{G} k(|x|,|y|)f(y)\,\mathrm{d}y\right)^p\mathrm{d}x\right)^{\frac{1}{p}}=|\mathfrak{S}|^{1+\frac{1}{p}}\left(\int_0^\infty\left(\int_0^\infty r^{\frac{Q-1}{p}}k(r,s)s^{Q-1-\frac{Q}{p}-\beta}\chi_{(1,\infty)}(s)\,\mathrm{d}s\right)^p \,\mathrm{d}r\right)^{\frac{1}{p}}.
\end{align*}
By using the sharpness in H$\ddot{\text{o}}$lder's inequality, we have
\begin{align}\label{G.H.K.Equ.4.9}
\left(\int_\mathbb{G}\left(\int_\mathbb{G} k(|x|,|y|)f(y)\,\mathrm{d}y\right)^p\mathrm{d}x\right)^{\frac{1}{p}}=|\mathfrak{S}|^{1+\frac{1}{p}}\sup\limits_{\|\psi\|_q=1} \int_0^\infty\int_1^\infty r^{\frac{Q-1}{p}}k(r,s)s^{Q-1-\frac{Q}{p}-\beta}\psi(r)\,\mathrm{d}s \,\mathrm{d}r\nonumber\\
\geq |\mathfrak{S}|^{1+\frac{1}{p}} \int_0^\infty\int_1^\infty r^{\frac{Q-1}{p}}k(r,s)s^{Q-1-\frac{Q}{p}-\beta}\psi(r)\,\mathrm{d}s \,\mathrm{d}r:=I,
\end{align}
for $\psi(r)=\left(\beta p\right)^{\frac{1}{q}}r^{-\frac{1}{q}(\beta p+1)}\chi_{(1,\infty)(r)}$, where $\frac{1}{p}+\frac{1}{q}=1$.

Next we use Fubini's theorem and obtain that
\begin{align}\label{G.H.K.Equ.4.10}
I&=|\mathfrak{S}|^{1+\frac{1}{p}}\left(\beta p\right)^{\frac{1}{q}} \int_1^\infty r^{\frac{Q}{p}-\beta(p-1)-1-Q}\left(\int_1^\infty k\left(1,\frac{s}{r}\right) s^{Q-1-\frac{Q}{p}-\beta}\,\mathrm{d}s\right)\,\mathrm{d}r\nonumber\\
&=|\mathfrak{S}|^{1+\frac{1}{p}}\left(\beta p\right)^{\frac{1}{q}} \int_1^\infty r^{-\beta p-1}\left(\int_{\frac{1}{r}}^\infty k(1,t) t^{Q-1-\frac{Q}{p}-\beta}\,\mathrm{d}t\right)\,\mathrm{d}r\nonumber\\
&\geq |\mathfrak{S}|^{1+\frac{1}{p}}\left(\beta p\right)^{\frac{1}{q}} \int_0^\infty k(1,t) t^{Q-1-\frac{Q}{p}-\beta}\left(\int_{\max{\{1, \frac{1}{t}\}}}^\infty r^{-\beta p-1} \,\mathrm{d}r \right)\,\mathrm{d}t\nonumber\\
&=|\mathfrak{S}|^{1+\frac{1}{p}}\dfrac{\left(\beta p\right)^{\frac{1}{q}}}{\beta p} \int_0^\infty k(1,t) t^{Q-1-\frac{Q}{p}-\beta}\left( \max\Big\{1, \frac{1}{t}\Big\}\right)^{-\beta p}\,\mathrm{d}t.
\end{align}
Hence, by using \eqref{G.H.K.Equ.4.8} -- \eqref{G.H.K.Equ.4.10} together with \eqref{G.H.K.Equ.4.3}, we can conclude that
\begin{align}\label{G.H.K.Equ.4.11}
|\mathfrak{S}| \int_0^\infty k(1,t) t^{Q-1-\frac{Q}{p}-\beta}\left( \max\Big\{1, \frac{1}{t}\Big\}\right)^{-\beta p}\,\mathrm{d}t\leq C<\infty.
\end{align}
Thus, by letting $\beta\to 0^+$ in \eqref{G.H.K.Equ.4.11} and using the Fatou's lemma, we see that \eqref{G.H.H.K.N.Equ.1.4} holds and even
\begin{align*}
C_p^\ast\leq C<\infty.
\end{align*}
The proof is complete including the fact that $C=C_p^\ast$ is the sharp constant in \eqref{G.H.K.Equ.4.3}.
\end{proof}
Next, we point out the following Euclidean version ($\mathbb{G}=\mathbb{R}$) of the above theorem (c.f. Theorem \ref{H.K.Thm.4.1}):
\begin{corollary}\label{G.H.K.Cor.1.2}
Let the kernel $k(|x|,|y|)$ be homogeneous of order $-1$ satisfy 
\begin{align}\label{G.H.K.Equ.4.12}
C_p^\ast=2\int_0^\infty k(1,s)s^{-\frac{1}{p}}\,\mathrm{d}s=2\int_0^\infty k(r,1)r^{-\frac{1}{q}}\,\mathrm{d}r,
\end{align}
where $p>1, \frac{1}{p}+\frac{1}{q}=1$. Then the inequality
\begin{align}\label{G.H.K.Equ.4.13}
\left(\int_\mathbb{R}\left(\int_\mathbb{R} k(|x|,|y|)f(y)\,\mathrm{d}y\right)^p\,\mathrm{d}x\right)^\frac{1}{p}\leq C\left(\int_\mathbb{R} f^p(x)\,\mathrm{d}x\right)^\frac{1}{p}
\end{align}
holds for all positive and measurable functions $f$ on $\mathbb{R}$ if and only if $C_p^\ast<\infty$. Moreover, the sharp constant in \eqref{G.H.K.Equ.4.13} is $C=C_p^\ast$, where $C_p^\ast$ is defined in \eqref{G.H.K.Equ.4.12}.
\end{corollary}

Next we state the following useful equivalence result of independent interest which holds for all kernels:

\begin{lemma}\label{G.H.K.Lemm.4.4}
Let $p>1, \dfrac{1}{p}+\dfrac{1}{q}=1$, $C$ be a constant, $0<C<\infty$, and $k(x,y)$ be a general kernel. Then the following inequalities are equivalent:
\begin{align}
\int_\mathbb{G}\int_\mathbb{G} k(x,y)f(x)g(y)\mathrm{d}x\mathrm{d}y&\leq C\left(\int_\mathbb{G} f^p(x)\mathrm{d}x\right)^\frac{1}{p}\left(\int_\mathbb{G} g^q(y)\mathrm{d}y\right)^\frac{1}{q}\label{G.H.K.Lem4.5.Equ.4.4}\\
\int_\mathbb{G}\left(\int_\mathbb{G} k(x,y)f(x)\mathrm{d}x\right)^p\mathrm{d}y&\leq C^p\int_\mathbb{G} f^p(x)\mathrm{d}x\label{G.H.K.Lem4.5.Equ.4.5}\\
\int_\mathbb{G}\left(\int_\mathbb{G} k(x,y)g(y)\mathrm{d}y\right)^q\mathrm{d}x&\leq C^q\int_\mathbb{G} g^q(y)\mathrm{d}y.\label{G.H.K.Lem4.5.Equ.4.6}
\end{align}
\end{lemma}
\begin{proof}
Suppose that \eqref{G.H.K.Lem4.5.Equ.4.4} holds. Let $g(y)=\left(\int_\mathbb{G} k(x,y)f(x)\mathrm{d}x\right)^{p-1}$. Then, by H$\ddot{\text{o}}$lder's inequality we have that
\begin{align*}
\int_\mathbb{G}\left(\int_\mathbb{G} k(x,y)f(x)\mathrm{d}x\right)^p\mathrm{d}y&=\int_\mathbb{G}\int_\mathbb{G} k(x,y)f(x)g(y)\mathrm{d}x\mathrm{d}y\\
&\leq C\left(\int_\mathbb{G} f^p(x)\mathrm{d}x\right)^\frac{1}{p}\left(\int_\mathbb{G} g^q(y)\mathrm{d}y\right)^\frac{1}{q}\\
&= C\left(\int_\mathbb{G} f^p(x)\mathrm{d}x\right)^\frac{1}{p}\left(\int_\mathbb{G} \left(\int_\mathbb{G} k(x,y)f(x)\mathrm{d}x\right)^p\,\mathrm{d}y\right)^\frac{1}{q},
\end{align*}
so \eqref{G.H.K.Lem4.5.Equ.4.5} holds because $1-\frac{1}{q}=\frac{1}{p}$. Similarly, one can show that \eqref{G.H.K.Lem4.5.Equ.4.4} implies \eqref{G.H.K.Lem4.5.Equ.4.6}. Conversely, assume that \eqref{G.H.K.Lem4.5.Equ.4.6} holds. Then, by H$\ddot{\text{o}}$lder's inequality and \eqref{G.H.K.Lem4.5.Equ.4.6}, we have that
\begin{align*}
\int_\mathbb{G}\int_\mathbb{G} k(x,y)f(x)g(y)\mathrm{d}x\mathrm{d}y &\leq\left( \int_\mathbb{G}\left(\int_\mathbb{G} k(x,y)g(y)\mathrm{d}y\right)^q\mathrm{d}x\right)^\frac{1}{q}\left(\int_\mathbb{G} f^p(x)\mathrm{d}x\right)^\frac{1}{p}\\
&\leq C\left(\int_\mathbb{G} g^q(y)\mathrm{d}y\right)^\frac{1}{q} \left(\int_\mathbb{G} f^p(x)\mathrm{d}x\right)^\frac{1}{p},
\end{align*}
so the inequality \eqref{G.H.K.Lem4.5.Equ.4.4} holds. Similarly we can prove that \eqref{G.H.K.Lem4.5.Equ.4.5} implies \eqref{G.H.K.Lem4.5.Equ.4.4} so the proof is complete.
\end{proof}
Finally, by just combining Theorem \ref{G.H.K.Thm.4.1} with Lemma \ref{G.H.K.Lemm.4.4} we obtain the following generalization of the equivalence Theorem \ref{G.H.K.Thm.4.6}, which connects Hardy-type and Hardy-Hilbert type inequalities with general kernels:
\begin{theorem}\label{G.H.K.Thm.4.6}
Let $p>1, \dfrac{1}{p}+\dfrac{1}{q}=1$, the kernel $k(|x|,|y|)$ be homogeneous of order $-Q$ and the constant $C_p^\ast$ be defined by \eqref{G.H.K.C.Equ.4.2}. Then the following four statements are equivalent:
\begin{enumerate}
\item[(i)] The constant $C_p^\ast <\infty$.
\item[(ii)] The Hardy-Hilbert type inequality
\begin{align}\label{G.H.K.C.Equ.4.17}
\int_\mathbb{G}\int_\mathbb{G} k(|x|,|y|)f(x)g(y)\mathrm{d}x\mathrm{d}y\leq C\left(\int_\mathbb{G} f^p(x)\mathrm{d}x\right)^\frac{1}{p}\left(\int_\mathbb{G} g^q(y)\mathrm{d}y\right)^\frac{1}{q}
\end{align}
holds for some finite constant $C>0$.
\item[(iii)] The Hardy-type inequality 
\begin{align}\label{G.H.K.C.Equ.4.18}
\int_\mathbb{G}\left(\int_\mathbb{G} k(|x|,|y|)f(x)\mathrm{d}x\right)^p\mathrm{d}y\leq C^p\int_\mathbb{G} f^p(x)\mathrm{d}x
\end{align}
 holds for the same finite constant $C>0$.
\item[(iv)] The inequality 
\begin{align}\label{G.H.K.C.Equ.4.19}
\int_\mathbb{G}\left(\int_\mathbb{G} k(|x|,|y|)g(y)\mathrm{d}y\right)^q\mathrm{d}x&\leq C^q\int_\mathbb{G} g^q(y)\mathrm{d}y
\end{align}
holds for the same finite constant $C>0$.
\end{enumerate}
Moreover, the constant $C_p^\ast$ is sharp in all of \eqref{G.H.K.C.Equ.4.17} -- \eqref{G.H.K.C.Equ.4.19}.
\end{theorem}

\begin{proof}
By using a similar arguments as in the proof of Theorem \ref{G.H.K.Thm.4.1} and Lemma \ref{G.H.K.Lemm.4.4} we finish the proof, so we omit the details. The proof is complete. 
\end{proof}

In particular, we present the following Euclidean version which also is new in this generality:
\begin{corollary}\label{G.H.H.Corr.5.6}
Let $p>1, \dfrac{1}{p}+\dfrac{1}{q}=1$, the kernel $k(|x|,|y|)$ be homogeneous of order $-n$ and the constant $C_p^\ast$ be defined by \eqref{G.H.K.C.Equ.4.2}. Then the following four statements are equivalent:
\begin{enumerate}
\item[(i)] The constant $C_p^\ast <\infty$.
\item[(ii)] The Hardy-Hilbert type inequality
\begin{align}\label{G.H.K.C.Equ.5.18}
\int_{\mathbb{R}^n}\int_{\mathbb{R}^n} k(|x|,|y|)f(x)g(y)\mathrm{d}x\mathrm{d}y\leq C\left(\int_{\mathbb{R}^n} f^p(x)\mathrm{d}x\right)^\frac{1}{p}\left(\int_{\mathbb{R}^n} g^q(y)\mathrm{d}y\right)^\frac{1}{q}
\end{align}
holds for some finite constant $C>0$.
\item[(iii)] The Hardy-type inequality 
\begin{align}\label{G.H.K.C.Equ.5.19}
\int_{\mathbb{R}^n}\left(\int_{\mathbb{R}^n} k(|x|,|y|)f(x)\mathrm{d}x\right)^p\mathrm{d}y\leq C^p\int_{\mathbb{R}^n} f^p(x)\mathrm{d}x
\end{align}
 holds for the same finite constant $C>0$.
\item[(iv)] The inequality 
\begin{align}\label{G.H.K.C.Equ.5.20}
\int_{\mathbb{R}^n}\left(\int_{\mathbb{R}^n} k(|x|,|y|)g(y)\mathrm{d}y\right)^q\mathrm{d}x&\leq C^q\int_{\mathbb{R}^n} g^q(y)\mathrm{d}y
\end{align}
holds for the same finite constant $C>0$.
\end{enumerate}
Moreover, the constant $C=C_p^\ast$ is sharp in all of \eqref{G.H.K.C.Equ.5.18} -- \eqref{G.H.K.C.Equ.5.20}.
\end{corollary}

\section*{Acknowledgments}
The first author thanks Ghent Analysis \& PDE Centre at Ghent University, Belgium. He is grateful to the centre for the support and warm hospitality during his long-term research visit. The research activities of the third author are funded by the FWO Odysseus 1 grant no. G.0H94.18N: Analysis and Partial Differential Equations, by the Methusalem programme of the Ghent University Special Research Fund (BOF) (grant no. 01M01021) and by the EPSRC (grants no. EP/R003025/2 and EP/V005529/1).



\end{document}